\documentclass[12pt]{article}
\usepackage[dvips]{graphicx}
\usepackage{amsmath,amsfonts,amssymb,amsthm,color}
\usepackage{natbib}

\newtheorem{theorem}{Theorem}
\newtheorem{lemma}{Lemma}

\newtheorem{proposition}{Proposition}

\newtheorem{definition}{Definition}
\newtheorem{example}{Example}

\usepackage{bm}
\bmdefine{\Bt}{t}
\bmdefine{\BX}{X}
\bmdefine{\BY}{Y}
\bmdefine{\BZ}{Z}
\bmdefine{\BB}{B}
\bmdefine{\BM}{M}
\bmdefine{\BD}{D}
\bmdefine{\Bi}{i}
\bmdefine{\Bj}{j}
\bmdefine{\Bx}{x}
\bmdefine{\By}{y}
\bmdefine{\Bz}{z}
\bmdefine{\Bv}{v}
\bmdefine{\Bw}{w}
\bmdefine{\Bn}{n}
\bmdefine{\Ba}{a}
\bmdefine{\Bb}{b}
\bmdefine{\Bc}{c}
\bmdefine{\Be}{e}
\bmdefine{\Bu}{u}
\bmdefine{\Bp}{p}
\bmdefine{\Bzero}{0}
\bmdefine{\Bone}{1}

\title{Markov Bases for Two-way Subtable Sum Problems}
\author{
Hisayuki Hara\\
Department of Geosystem Engineering\\
University of Tokyo \smallskip\\
Akimichi Takemura\\
Graduate School of Information Science and Technology\\
University of Tokyo\\
and\\
Ruriko Yoshida\\
Department of Statistics\\
University of Kentucky}
\date{August 2007}

\begin{document}
\maketitle

\begin{abstract}
 It has been well-known that for two-way contingency tables with fixed
 row sums and column sums the set of square-free moves of degree
 two forms a Markov basis.
 However when we impose an additional constraint that
 the sum of a subtable is also fixed, then these moves do not necessarily
 form a Markov basis.
 Thus, in this paper, we show a necessary and sufficient condition on a 
 subtable so that 
 the set of square-free moves of degree two forms a Markov basis.
\end{abstract}

\section{Introduction}
Since 
\citet{sturmfels1996} and 
\cite{diaconis-sturmfels}  showed that a set of binomial generators of a 
toric ideal for a statistical model of discrete exponential families is
equivalent to a Markov basis 
and initiated Markov chain Monte Carlo approach based on a Gr\"obner 
basis computation for testing statistical fitting of the given model, 
many researchers have extensively studied the
structure of Markov bases for models in computational
algebraic statistics (e.g.  
\cite{hosten-sullivant, dobra-2003bernoulli, dobra-sullivant,
geiger-meek-sturmfels}).  

In this article we consider Markov bases for two-way contingency
tables with fixed row sums, column sums and an additional
constraint that the sum of a subtable is also fixed.
We call this problem a {\em two-way subtable sum problem}.
From statistical viewpoint this problem is motivated  
by a block interaction model or a two-way change-point model
proposed by 
\cite{hirotsu-1997}, which has been studied from both
theoretical and practical viewpoint(\cite{ninomiya-2004}) and has 
important applications to dose-response clinical trials with ordered categorical
responses.  

Our model also relates to the quasi-independence model for incomplete
two-way contingency tables which contain some structural zeros
(\cite{aoki-takemura-2005jscs}, \cite{rapallo2006}).
Essentially the same problem 
has been studied in detail 
from algebraic viewpoint 
in a series of papers by Ohsugi and Hibi
(\cite{ohsugi-hibi-1999aam}, \cite{ohsugi-hibi-1999ja}, \cite{ohsugi-hibi-2005jaa}).

It has been well-known that for two-way contingency tables with fixed
row sums and column sums the set of square-free moves of degree
two forms a Markov basis.  
However when we impose an additional constraint that
the sum of a subtable is also fixed, then these moves do not necessarily
form  a Markov basis.  
\begin{example}\label{eg}
Suppose we have a $3 \times 3$ table with the following cell counts.
\[
\begin{array}{|c|c|c|}\hline
7 & 5 & 1 \\ \hline
5 & 10 & 6\\ \hline
2 & 6 & 8 \\ \hline
\end{array}\, .
\]
If we fix the row sums $(13, 21, 16)$ and column sums $(14, 21, 15)$, and
also if we fix the sum of two cells at $(1,1)$ and $(2,1)$$($$7+5=12$ in
this  example$)$, 
a Markov basis consists of square-free moves of degree
two.   However, if we fix the sum of 
two cells at $(1,1)$ and $(2,2)$$($$7+10=17$ in this example$)$, 
then a Markov basis contains non-square-free moves such as
\[
\begin{array}{|c|c|c|}\hline
 1 & 1 &-2\\ \hline
-1& -1&  2\\ \hline
0& 0 &0\\ \hline
\end{array}\, .
\]
\end{example}
In this paper we show a necessary and sufficient
condition on a subtable so that
a corresponding Markov basis consists of
square-free moves of degree two.
The results here may give some insights into Markov bases 
for statistical models for general multi-way tables
with various patterns of statistical interaction effects.

Because of the equivalence between a Markov basis and a set of binomial
generators of a toric ideal, a theory of this paper can be
entirely translated and developed in an algebraic framework.  However,
in this paper we make extensive use of pictorial representations of
tables and moves.  Therefore we prefer to develop our theory using
tables and moves. 
See \cite{math.st/0511290}
for a discussion of this equivalence.

This paper is organized as follows:
In section \ref{Preliminaries}, we describe our problem and summarize some preliminary
facts. Section \ref{necsufcond} gives a necessary and sufficient condition that 
a Markov basis consists of square-free moves of degree two. 
We end this paper with some concluding remarks in Section \ref{sec:discussion}.

\section{Preliminaries}\label{Preliminaries}
\subsection{Subtable sum problem and its Markov bases}
Let $\mathbb{N} = \{0,1,2,\ldots\}$ and let $X=\{x_{ij}\}$, 
$x_{ij} \in \mathbb{N}$, $i=1,\ldots,R$, $j=1,\ldots,C$, 
be an 
$R \times C$ 
table with nonnegative integer entries. 
Let ${\cal I} = \{(i,j) \mid 1 \le i \le R, 1 \le j \le C\}$. 
Using statistical terminology, we call $X$ a {\em contingency table} and $\cal I$
the set of {\em cells}.

Denote the row sums and column sums of $X$ by
\[
x_{i+}=\sum_{j=1}^C x_{ij}, \quad i=1,\ldots,R, \qquad
    x_{+j}=\sum_{i=1}^R x_{ij}, \quad  j=1,\ldots,C.
\]
Let $S$ be a subset of ${\cal I}$. Define the subtable sum $x(S)$ by 
\[
 x(S) = \sum_{(i,j) \in S} x_{ij}.
\]
Denote the set of row sums, column sums and $x(S)$ by
$$
\bm{b}=\{x_{1+},\ldots,x_{R+},x_{+1},\ldots,x_{+C},x(S)\}.
$$
For $S=\emptyset$ or $S={\cal I}$, we have
$x(\emptyset)\equiv 0$
or $x({\cal I})=\deg X :=\sum_{(i,j)\in {\cal I}}x_{ij}=\sum_i x_{i+}$.  
In these cases $x(S)$ is redundant and
our problem reduces to a problem concerning tables with fixed row sums and
column sums.  Therefore in the rest of this paper, we consider $S$
which is a non-empty proper subset of $\cal I$.
Also note that $x(S^c)= \deg X - x(S)$, where 
$S^c$ is the complement of $S$.  Therefore fixing $x(S)$ 
is equivalent to fixing $x(S^c)$.

We consider $\bm{b}$ as a column vector with dimension $R+C+1$.
We also order the elements of $X$ with respect to a lexicographic order
and regard $X$ as a column vector with dimension $|{\cal I}|$.
Then the relation between $X$ and $\bm{b}$ is written by 
\begin{equation}
 \label{configulation}
  A_S X = \bm{b}. 
\end{equation}
Here $A_S$ is an $(R+C+1) \times |{\cal I}|$ matrix
consisting of $0$'s and $1$'s.
The set of columns of $A_S$ is a configuration defining a toric ideal $I_{A_S}$.
In this paper we simply call $A_S$ the {\em configuration} for $S$.
The set of tables $X\in {\mathbb N}^{\cal I}$
satisfying (\ref{configulation}) is called 
the {\em fiber} for $\bm{b}$  and is denoted by ${\cal F}(\bm{b})$.

An $R \times C$ integer array $B=\{b_{ij}\}_{(i,j) \in {\cal I}}$ 
satisfying 
\begin{equation}
 \label{move}
  A_S B=\bm{0}
\end{equation}
is called a {\it move} for the configuration $A_S$.  Let
\[
{\cal M}_S=   \{ B \mid  A_S B=\bm{0}\}
\]
denote the  set of moves for $A_S$.  
Let ${\cal B}\subset {\cal M}_S$ be a subset of ${\cal M}_S$.
Note that if $B$ is a move then $-B$ is a move.  We call ${\cal B}$
sign-invariant if $B\in {\cal B} \Rightarrow -B \in {\cal B}$. 
According to \cite{diaconis-sturmfels},  
a Markov basis for $A_S$ is equivalent to 
a set of binomial generators of the corresponding toric ideal for $I_{A_S}$
and defined as follows. 

\begin{definition}
 \label{def:MB}
 A Markov basis for $A_S$ is a sign-invariant finite set of moves 
 ${\cal B}=\{B_1,\ldots,B_L\} \allowbreak \subset {\cal M}_S$ such that, 
 for any $\bm{b}$ and $X$, $Y \in {\cal F}(\bm{b})$, there exist
 $\alpha > 0$, 
$B_{t_1},\ldots,B_{t_{\alpha}} \in {\cal B}$ 
such that 
 $$
 Y = X + \sum_{s=1}^{\alpha} B_{t_s} 
 \quad \text{and} \quad 
 Y = X + \sum_{s=1}^a B_{t_s} 
 \in {\cal F}(\bm{b}) \quad \text{for\;}
 1 \le a \le \alpha.
 $$
\end{definition}

In this paper, for simplifying notation and without loss of
generality, we only consider sign-invariant sets of moves as Markov bases.

For $i \neq  i'$ and $j \neq j'$, 
consider the square-free move of degree two with $+1$ at cells $(i,j)$, $(i',j')$ and $-1$
at cells $(i,j')$ and $(i',j)$ :
$$
\begin{array}{crr}
 & j & j' \\
 i &  1 & -1 \\
 i'& -1 &  1
\end{array}
$$
For simplicity we call this 
a {\it basic move} and and denote it by 
$$
B(i,i';j,j')=(i,j)(i',j')-(i,j')(i',j).
$$
It is well-known that
the set of all basic moves 
$$
{\cal B}_0 = \{B(i,i';j,j') \mid (i,j) \in {\cal I}, 
(i',j') \in {\cal I}, i \neq  i', j \neq  j'\}
$$ 
forms a unique minimal Markov basis for $A_\emptyset$, i.e.\  the
problem concerning tables with fixed rows sums and column sums.
If $B(i,i';j,j')\in {\cal M}_S$, 
we call it a basic move for $S$. 
Define 
\[
{\cal B}_0(S) = {\cal B}_0 \cap {\cal M}_S
\]
which is the set of all basic moves for $S$.  Note that  ${\cal B}_0(S)$
coincides with the set of square-free moves of degree two for $A_S$,
since the row sums and columns sums are fixed.

As clarified in Section \ref{necsufcond}, ${\cal B}_0(S)$ does not
always form a Markov basis for $A_S$.
In Section \ref{necsufcond}, we derive a necessary and sufficient condition 
on $S$ that ${\cal B}_0(S)$ is a Markov basis.

\subsection{Reduction of $L_1$-norm of a move and Markov bases}
In proving that ${\cal B}_0(S)$ is a Markov basis for a given $S$, we
employ the norm-reduction argument of
\cite{takemura-aoki-2005bernoulli} and \cite{aoki-takemura-2003anz}.
Suppose that we have two tables $X$ and $Y$ in the same fiber ${\cal F}$. 
Denote 
$$
X - Y = \{x_{ij}-y_{ij}\}_{(i,j) \in {\cal I}} 
$$ 
and define the $L_1$-norm of $X-Y$ by 
$\Vert X-Y \Vert_1 = \sum_{(i,j) \in {\cal I}} \vert x_{ij} -
y_{ij}\vert$.  
We define that $\Vert X-Y \Vert_1$ can be reduced (in several
steps) by  ${\cal B}_0(S)$ as follows.

\begin{definition}
 \label{def:reduce}
For $X \neq Y$ in the same fiber $\cal F$, we say
that $\Vert X-Y \Vert_1$ can be reduced by ${\cal B}_0(S)$ 
if there exist $\tau^+ \ge 0,  \tau^- \ge 0$, $\tau^+ +\tau^->0$,  and 
sequences of moves 
$B^+_t \in {\cal B}_0(S)$, $t=1,\ldots, \tau^+$, 
and 
$B^-_t \in {\cal B}_0(S)$, 
$t=1,\ldots, \tau^-$, 
satisfying
\begin{equation}
   \label{ineq:L_1_norm}
   \left\{
    \begin{array}{c}
     \displaystyle{
      \Vert X - Y + \sum_{t=1}^{\tau^+} B^+_t
      + \sum_{t=1}^{\tau^-} B^-_t
      \Vert_1 < \Vert X - Y \Vert_1, }\\
     \displaystyle{
      X + \sum_{t=1}^{\tau'}  B^+_t \in {\cal F}, \quad \text{for} \quad
      \tau'=1,\ldots,\tau^+,}\\
     \displaystyle{
      Y - \sum_{t=1}^{\tau'} B^-_t \in {\cal F}, \quad \text{for} \quad
      \tau'=1,\ldots,\tau^-} .
    \end{array}
	 \right.
\end{equation}
\end{definition}

In \cite{takemura-aoki-2005bernoulli} we have mainly considered the
case that $\Vert X-Y \Vert_1$ can be reduced in  one step: $\tau^+ +
\tau^-=1$.  However as discussed in Section 4.2 of
\cite{takemura-aoki-2005bernoulli}, it is clear that ${\cal B}_0(S)$
is a Markov basis for $A_S$ if for every fiber ${\cal F}(\bm{b})$ and
for every $X\neq  Y$ in ${\cal F}(\bm{b})$, 
$\Vert X-Y \Vert_1$ can always be 
reduced by ${\cal B}_0(S)$.  
Here the number of steps $\tau^+ + \tau^- $ needed to reduce 
$\Vert X-Y \Vert_1$ can depend on $X$ and $Y$.  Therefore we
consider a condition that $\Vert X - Y \Vert_1$ can be reduced by
${\cal B}_0(S)$.

As in \cite{aoki-takemura-2003anz}, we look at the patterns of the
signs of $X-Y$.
Suppose that $X - Y$ has the pattern of signs as in 
Figure \ref{figure:pattern-2x2-1}-(i). 
This means 
\[
x_{i'j} < y_{i'j}, \quad x_{ij'} < y_{ij'} 
\]
and the signs of $x_{ij} - y_{ij}$ and  $x_{i'j'} - y_{i'j'}$ are
arbitrary.
Henceforth let $*$ represent that the sign of the cell is arbitrary
as in Figure \ref{figure:pattern-2x2-1}.
Because $x_{i'j} \ge 0$, $x_{ij'} \ge 0$, 
we have
\[
y_{i'j} > 0, \quad y_{ij'} > 0.
\]
Therefore for $B^- = (i,j)(i',j')-(i,j')(i',j) \in {\cal B}_0(S)$, 
we have $Y - B^- \in {\cal F}$ and 
we note that $\Vert X - Y + B^- \Vert_1 \le \Vert X - Y \Vert_1$
regardless of the signs of 
$x_{ij} - y_{ij}$ and $x_{i'j'} - y_{i'j'}$.  
If 
$x_{ij} \le y_{ij}$ and $x_{i'j'} \le y_{i'j'}$, 
$$
\Vert X - Y + B^- \Vert_1 = \Vert X - Y \Vert_1.
$$
On the other hand, if
$x_{ij} > y_{ij}$ or $x_{i'j'} > y_{i'j'}$, i.e. 
$X - Y$ has the pattern of signs as in 
Figure \ref{figure:pattern-2x2-2}-(i) or (ii), 
we have 
\begin{equation}
 \label{ineq:L_1_norm-2}
  \Vert X - Y + B^-\Vert_1 < \Vert X - Y \Vert_1.
\end{equation}
In this case $\tau^+=0$,  $\tau^-=1$ and $B_1^-=B^-$
satisfy (\ref{ineq:L_1_norm}).
By interchanging the role of $X$ and $Y$, we can see that 
the patterns in (i) and (ii) in Figure \ref{figure:pattern-2x2-1}
are interchangeable. 
Hence similar argument can be done for the patterns (ii) in 
Figure \ref{figure:pattern-2x2-1} and (iii), (iv) in
Figure \ref{figure:pattern-2x2-2}. 
\begin{figure}[t]
 \centering
$$
\begin{array}{ccc}
 & j & j' \\
 i & * & - \\
 i'& - & *
\end{array}
\quad
\begin{array}{ccc}
 & j & j' \\
 i & * & + \\
 i'& + & *
\end{array}
$$
\hspace*{0.2cm}
(i)\hspace{1.8cm}(ii)
 \caption{Patterns of signs in a $2 \times 2$ subtable}
\label{figure:pattern-2x2-1}
\end{figure}
\begin{figure}[t]
 \centering
$$
\begin{array}{ccc}
 & j & j' \\
 i & + & - \\
 i'& - & *
\end{array}
\quad
\begin{array}{ccc}
 & j & j' \\
 i & * & - \\
 i'& - & +
\end{array}
\quad
\begin{array}{ccc}
 & j & j' \\
 i & - & + \\
 i'& + & *
\end{array}
 \quad
\begin{array}{ccc}
 & j & j' \\
 i & * & + \\
 i'& + & -
\end{array}
$$
\hspace*{0.2cm}
(i)\hspace{1.8cm}(ii)\hspace{1.8cm}(iii)\hspace{1.7cm}(iv)
 \caption{Patterns of signs in a $2 \times 2$ subtable}
\label{figure:pattern-2x2-2}
\end{figure}

Denote $Z = Z_0 = X-Y$. 
For a sequence of basic moves $B_t\in {\cal B}_0(S)$,
$t=1,\ldots,\tau$ 
denote 
$ Z_t = X - Y +  B_1 + \cdots + B_t$,  
$t=1,\ldots,\tau$. 
Based on the above arguments, we obtain the following lemma.  The
proof is easy and omitted.

\begin{lemma}
\label{lemma:reduction}
$\Vert Z \Vert_1$ can be reduced by ${\cal B}_0(S)$ if there exist
$\tau > 0$ and a sequence of basic moves 
$B_t\in {\cal   B}_0(S)$, 
$t=1,\ldots,\tau$ such that 
$Z_t$, $t = 0 \ldots, \tau-1$, have either of the sign
patterns in Figure \ref{figure:pattern-2x2-1} and 
$Z_{\tau}$ has either of the patterns in Figure
\ref{figure:pattern-2x2-2}.
\end{lemma}

This lemma will be repeatedly used from Section
\ref{sec:blockwise-signs} on.

\section{A necessary and sufficient condition}
\label{necsufcond}
In this section we give a necessary and sufficient condition on 
the subtable sum problem so that a Markov basis consists of basic moves,
i.e.\ ${\cal B}_0(S)$ forms a Markov basis for $A_S$.
Figure \ref{figure:pattern_P} shows patterns of $2 \times 3$ and 
$3\times 2$ tables. A shaded area shows a cell belonging to $S$.
Henceforth let a shaded area represent a cell belonging to $S$ or 
rectangular blocks of cell belonging to $S$.
We call these two patterns in Figure \ref{figure:pattern_P} 
the pattern ${\cal P}$ and ${\cal P}^t$, respectively.
Then a necessary and sufficient condition is expressed as follows.

\begin{figure}[htbp]
 \centering
 \includegraphics{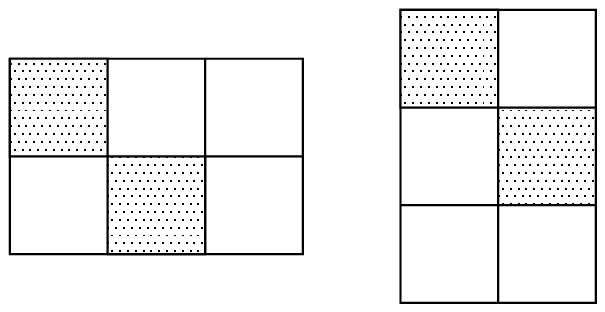}\\
 \hspace*{0.5cm}${\cal P}$\hspace{3cm}${\cal P}^t$\\
 \caption{The pattern ${\cal P}$ and ${\cal P}^t$}
\label{figure:pattern_P}
\end{figure}

\begin{theorem} 
 \label{Th:main}
 ${\cal B}_0(S)$ is a Markov basis for $A_S$
 if and only if 
 there exist no patterns of the form
 ${\cal P}$ or ${\cal P}^t$ in any $2 \times 3$ and $3 \times 2$ 
 subtable of $S$ or $S^c$ after any interchange of rows and columns. 
\end{theorem}

We give a proof of 
Theorem \ref{Th:main} in the following subsections. 
Note that if ${\cal B}_0(S)$ is a Markov basis for $A_S$, then it
is the {\em unique minimal Markov basis}, because the basic 
moves in ${\cal B}_0(S)$ are all indispensable.  

The outline of this section is as follows.
Section \ref{subsec3:1} gives a proof of the necessary condition.
In Section \ref{subsec3:2} we introduce two patterns of $S$, $2 \times 2$ block
diagonal set and triangular set, and show that 
$S$ or $S^c$ contain patterns of the form ${\cal P}$ or ${\cal P}^t$ if
and only if $S$ is equivalent to either of the two patterns.
Then the sufficiency can be rewritten that 
${\cal B}_0(S)$ forms a Markov basis for $S$ equivalent to a $2 \times
2$ block diagonal set or a triangular set. 
In Section \ref{subsec3:3} we prepare some ingredients to prove the sufficiency.
In Section \ref{subsec3:4} and Section \ref{subsec3:5} we show proofs of the sufficient condition for
$2 \times 2$ block diagonal set and triangular set, respectively.

\subsection{A proof of the necessary condition}\label{subsec3:1}
The necessary condition of Theorem 
\ref{Th:main} is easy to prove.

\begin{proposition}
 \label{prop:necessity}
 If $S$ or $S^c$ contains the pattern ${\cal P}$ or ${\cal P}^t$, 
 ${\cal B}_0(S)$ is not a Markov basis for $A_S$. 
\end{proposition}

\begin{proof}
Assume that $S$ has the pattern ${\cal P}$.
Without loss of generality we can assume that 
${\cal P}$ belongs to $\{(i,j) \mid i=1,2,\; j=1,2,3\}$. 
Consider a fiber such that 
\begin{itemize}
 \item $x_{1+} = x_{2+} = 2$, $x_{+1} = x_{+2} = 1$, $x_{+3} = 2$;
 \item $x_{i+}=0$ and $x_{+j}=0$ for all 
       $(i,j) \notin \{(i,j) \mid i=1,2,\; j=1,2,3\}$;
 \item $\sum_{(i,j) \in S} x_{ij} = 1$;
\end{itemize}

Then it is easy to check that this fiber has only two elements 
shown in Figure \ref{figure:fiber}. 
Hence the difference of these two tables
\begin{equation}
\label{eq:principal-ideal}
B = 
\begin{array}{|r|r|r|} \hline
 1 & 1 & -2\\ \hline
 -1 & -1 & 2\\ \hline
\end{array}
\end{equation}
is an indispensable move. 
Therefore if $S$ has the pattern ${\cal P}$, 
there does not exist a Markov basis consisting of basic moves.
When $S$ has the pattern ${\cal P}^t$, a proof is
similar. 

\end{proof}

\begin{figure}[htbp]
 \centering
 \includegraphics{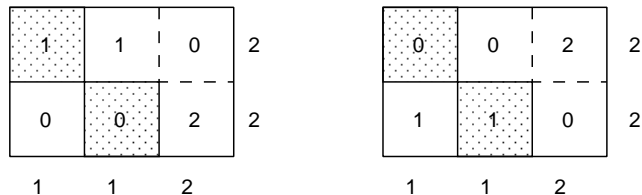}
 \caption{Two elements of the fiber}
\label{figure:fiber}
\end{figure}

It is of interest to note that the toric ideal for the $2\times 3$
table with the pattern ${\cal P}$ of $S$ is a principal ideal generated by a
single binomial corresponding to (\ref{eq:principal-ideal}) whose both
monomials are non-square-free.

\subsection{Block diagonal sets and triangular sets}\label{subsec3:2}
\label{sec:block-diagonal-trangular}
After an appropriate interchange of rows and columns, 
if $S$ satisfies that 
$$
S = \{(i,j) \mid i \le r, j \le c\} \cup
\{(i,j) \mid i > r, j > c\}
$$
for some $r < R$ and $c < C$, 
we say that $S$ is equivalent to a $2 \times 2$ block diagonal set.
Figure \ref{figure:2-block} shows a $2 \times 2$ block diagonal set.
A $2 \times 2$ block diagonal set is decomposed into four blocks
consisting of one or more cells.
We index each of the four blocks  as in Figure \ref{figure:2-block}. 
Note that $S$ is a  $2 \times 2$ block diagonal set if and only if 
$S^c$ is a  $2 \times 2$ block diagonal set.

For a row index $i$, 
let ${\cal J}(i)  = \{ j | (i,j) \in S\}$ denote a slice of 
$S$ at row $i$.    
If for every pair $i$ and $i'$,  
either ${\cal J}(i)$ is a subset of ${\cal J}(i')$ or 
${\cal J}(i)$ is a superset of ${\cal J}(i')$,  
we say that $S$ is equivalent to a {\em triangular set}.
A triangular set is expressed as in Figure \ref{figure:triangular}
after an appropriate interchange of rows and columns.
In general, if we allow transposition of tables, 
triangular sets can be decomposed into 
$n \times (n+1)$ or $n \times n$ blocks as in Figure
\ref{figure:triangular}.  
Figure \ref{figure:triangular} shows examples of $n \times (n+1)$ and 
$n \times n$ triangular sets with $n=4$. 
We index each block as in Figure \ref{figure:triangular}. 
Let ${\cal F}^T $ be a fiber of an $n \times (n+1)$ triangular set. 
Define ${\cal J}_{n+1} = \{j \mid (i,j) \in {\cal I}_{k,n+1}\}$. 
Then we note that if ${\cal F}^T$ satisfies 
$\sum_{i=1}^R x_{ij} = 0$ for all $j \in {\cal J}_{n+1}$, 
the fiber is equivalent to a fiber for an $n \times n$ triangular set.
Hence an $n \times n$ triangular set is interpreted as a special case of 
an $n \times (n+1)$ triangular set.   
Hereafter we consider only $n \times (n+1)$ triangular sets and let 
a triangular set mean $n \times (n+1)$ triangular set. 
Note also that $S$ is a triangular
set if and only if $S^c$ is a triangular set.
In other words, a triangular set is symmetric with respect to 
$180^\circ$ rotation of the table.

\begin{figure}[b]
 \centering
 \includegraphics{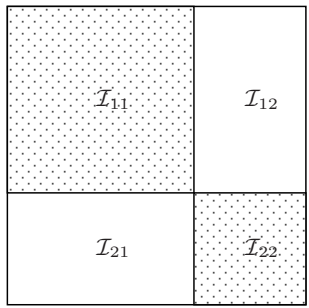}
 \caption{$2 \times 2$ block diagonal set}
 \label{figure:2-block}
\end{figure}

\begin{figure}[htbp]
 \centering
 \includegraphics{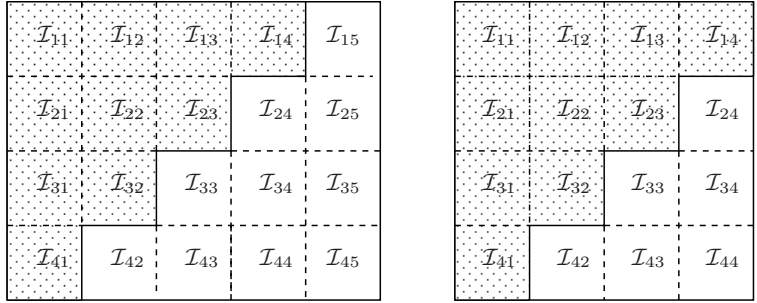}
 \caption{(block-wise) $4 \times 5$ and $4 \times 4$ triangular sets }
 \label{figure:triangular}
\end{figure}

\begin{proposition}
 \label{prop:2-block_triangular}
There exist no patterns of the form
${\cal P}$ or ${\cal P}^t$ in any $2 \times 3$ and $3 \times 2$ 
subtable of $S$ after any interchange of rows and columns  
if and only if 
$S$ is equivalent to a $2 \times 2$ block diagonal set or 
a triangular set.
\end{proposition}

\begin{figure}[htbp]
 \centering
 \includegraphics[scale=0.8]{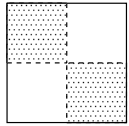}
 \caption{ $2 \times 2$ (cell-wise) crossing pattern}
\label{figure:cross}
\end{figure}

\begin{figure}[htbp]
 \centering
 \includegraphics[scale=0.65]{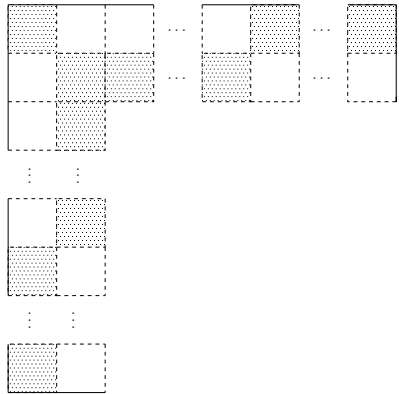}
 \caption{The pattern of  $\{(i,j) \mid i = 1,2\}$ and $\{(i,j) \mid j =
 1,2\}$}
 \label{figure:pattern-1}
\end{figure}

\begin{figure}[htbp]
 \centering
 \includegraphics[scale=0.65]{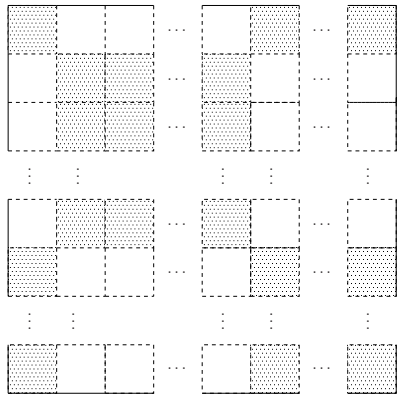}
 \caption{The pattern of $S$ which has a (cell-wise) $2 \times 2$ crossing pattern}
\label{figure:pattern-2}
\end{figure}

\begin{proof}
 Assume that $S$ does not contain ${\cal P}$ and ${\cal P}^t$ 
 and that $S$ contains a (cell-wise) 
 $2 \times 2$ crossing sub-pattern presented 
 in Figure \ref{figure:cross}.  
 Without loss of generality the crossing pattern belongs to 
 $\{(i,j) \mid i=1,2, j=1,2\}$. 
 Since $S$ does not contain ${\cal P}$ and ${\cal P}^t$,
 $\{(i,j) \mid i = 1,2\}$ and $\{(i,j) \mid j = 1,2\}$
 have to have the pattern as in Figure \ref{figure:pattern-1} 
 after an appropriate interchange of rows and columns. 
 In the same way the rest of the table 
 $\{(i,j) \mid i \ge 3, j \ge 3\}$ has to have the pattern as in 
 Figure \ref{figure:pattern-2}. 
 It is clear that the pattern in Figure \ref{figure:pattern-2} is
 equivalent to a $2 \times 2$ block diagonal pattern after 
 interchanging rows and columns. 
 
 From the definition of triangular set, $S$ is not equivalent to a
 triangular set if and only if there exists $i$, $i'$, 
 $i \neq i'$, and  $j$, $j'$, $j\neq j'$, such that
 $j \in {\cal J}(i)$, $j \notin {\cal J}(i')$, 
 $j' \in {\cal J}(i')$ and $j' \notin {\cal J}(i)$. 
 But this is equivalent to the existence of 
 a $2 \times 2$ crossing pattern.
\end{proof}

\subsection{Signs of blocks}\label{subsec3:3}
\label{sec:blockwise-signs}
Based on Proposition \ref{prop:2-block_triangular},   
for the sufficient condition of 
Theorem  \ref{Th:main}
we only need to show that ${\cal B}_0(S)$ forms a Markov basis
for $S$  equivalent to a $2 \times 2$ block
diagonal set or a triangular set.
As mentioned above, a $2 \times 2$ block diagonal set and a triangular set
can be decomposed into some rectangular blocks. 
In general each block consists of more than one cell.
For the rest of this section, we use the following lemma.

\begin{lemma}
 \label{lemma:2}
 Assume that $S$ is equivalent to a $2 \times 2$ block diagonal set or 
 a triangular set.
 Suppose that $Z = \{z_{ij}\}_{(i,j) \in {\cal I}}$ contains a block ${\cal I}_{kl}$ such that 
 $(i,j) \in {\cal I}_{kl}$, $(i',j') \in {\cal I}_{kl}$, 
 $z_{ij} > 0$ and $z_{i'j'} < 0$. 
 Then $\Vert Z \Vert_1$ can be reduced by ${\cal B}_0(S)$. 
\end{lemma}

\begin{proof}
 Suppose that $j=j'$ and $i\neq i'$.
 Since any row sum of $Z$ is zero, there exists $j''$ such that 
 $z_{ij''} < 0$ as presented in Figure \ref{figure:lemma2-1}.
 Hence $Z$ 
 contains the sign pattern of Figure \ref{figure:pattern-2x2-1}-(i)
 and $\Vert Z \Vert_1$ can be reduced by ${\cal B}_0(S)$.
 When $i=i'$ and $j \neq j'$,
 we can show that $\Vert Z \Vert_1$ can be reduced by ${\cal B}_0(S)$ in
 the similar way. 

 Next we consider the case where $i \neq i'$ and $j \neq j'$.
 If  $z_{i'j} \neq 0$ or $z_{ij'} \neq 0$, 
 we can reduce $\Vert Z \Vert_1$ by using the above argument regardless
 of the signs of them.
 So we suppose $z_{i'j}=0$ and $z_{ij'}=0$.
 There exists $j''$ such that $z_{ij''} < 0$ as 
 presented in Figure \ref{figure:lemma2-2}-(i). 
 If $(i,j'') \in {\cal I}_{kl}$, we can reduce $\Vert Z \Vert_1$
 by using the above argument.
 If $(i,j'') \notin {\cal I}_{kl}$, 
 let $B=(i,j'')(i',j')-(i,j')(i',j'') \in {\cal B}_0(S)$ and 
 let $Z' = \{ z'_{ij}\} = Z + B$.
 Since $z_{ij''} < 0$ and $z_{i'j'} < 0$, 
 we have 
 $\Vert Z' \Vert_1 
 \leq \Vert Z \Vert_1$. 
 We also have 
 $z'_{ij} > 0$ and $z'_{ij'} < 0$. 
 Since $(i,j),(i,j') \in {\cal I}_{kl}$, 
 $\Vert Z' \Vert_1$ can be reduced by ${\cal B}_0(S)$.
 Therefore $Z$ satisfies the condition of  Lemma \ref{lemma:reduction}
 and $\Vert Z \Vert_1$ can be reduced by ${\cal B}_0(S)$.
\end{proof}

 \begin{figure}[htbp]
  \centering
  \includegraphics[scale=0.7]{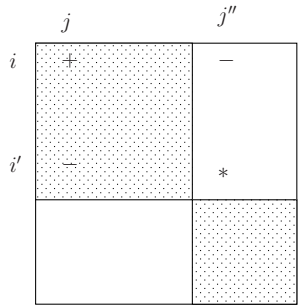}
  \caption{$Z$ when $j=j'$}
  \label{figure:lemma2-1}
 \end{figure}

 \begin{figure}[htbp]
  \centering
  \includegraphics[scale=0.7]{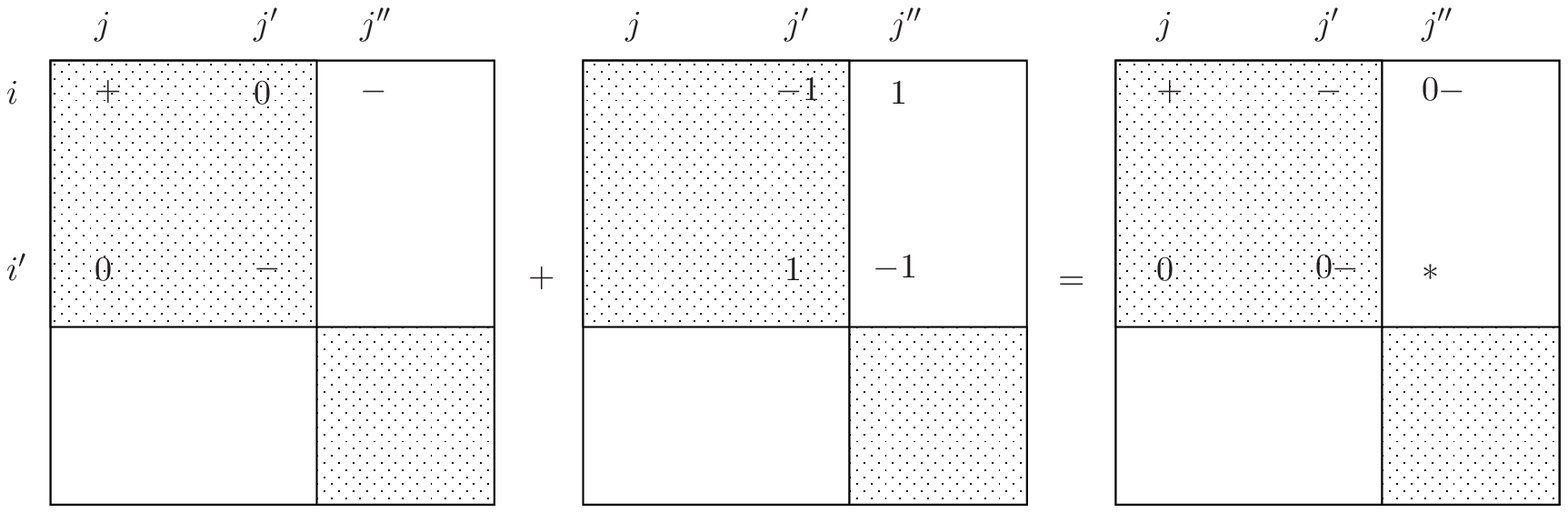}\\
  \hspace*{0.2cm}(i) $Z$ \hspace{2.8cm} (ii) $B$ \hspace{2.8cm} (iii) $Z'$
  \caption{$Z$ and $Z';$ when $j \neq j'$}
  (``$0-$'' represents the cell which is nonpositive)
  \label{figure:lemma2-2}
 \end{figure}

Let ``$0+$'' and ``$0-$'' represent the cell which is nonnegative and
nonpositive, respectively, as in Figure \ref{figure:lemma2-2}.

From Lemma \ref{lemma:2} in order to prove the sufficient condition
of Theorem \ref{Th:main}, we only need to 
consider the case where $Z$ does not have a block 
with both positive and negative cells. 
If all cells in ${\cal I}_{kl}$ are zeros, 
we denote it by ${\cal I}_{kl} = 0$.
If ${\cal I}_{kl} \neq  0$ and all nonzero cells in ${\cal I}_{kl}$ are positive, 
we denote it by ${\cal I}_{kl} > 0$ and we say 
${\cal I}_{kl}$ is positive.
We define ${\cal I}_{kl} < 0$, 
${\cal I}_{kl} \ge 0$ and ${\cal I}_{kl} \le 0$ in the similar way.
Then we say ${\cal I}_{kl}$ is negative, nonnegative and nonpositive,
respectively. 
Then we obtain the following lemma.

\begin{lemma}
 \label{lemma:triangular}
 Assume that $S$ is equivalent to a triangular set.
 Suppose that $Z$ has four blocks ${\cal I}_{kl}$, ${\cal I}_{k'l}$,  
 ${\cal I}_{kl'}$ and ${\cal I}_{k'l'}$ which have either of the patterns
 of signs as follows, 
\begin{center}
$ \begin{array}{ccc}
  & l & l' \\
  k & + & - \\
  k'& - & *
 \end{array}
 \quad
 \begin{array}{ccc}
  & l & l' \\
  k & * & - \\
  k'& - & +
 \end{array}
 \quad
 \begin{array}{ccc}
  & l & l' \\
  k & - & + \\
  k'& + & *
 \end{array}
 \quad
 \begin{array}{ccc}
  & l & l' \\
  k & * & + \\
  k'& + & -
 \end{array}
 $\\ \quad
{\rm (i)}\hspace{1.8cm}{\rm (ii)}\hspace{1.8cm}{\rm
  (iii)}\hspace{1.8cm}{\rm (iv)}\\
\end{center}
where $*$ represents that the sign of the block is arbitrary.
 If there exist $i,i',j,j'$ such that 
 $(i,j) \in {\cal I}_{kl}$, 
 $(i',j) \in {\cal I}_{k'l}$, 
 $(i,j') \in {\cal I}_{kl'}$ 
 $(i',j') \in {\cal I}_{k'l'}$, 
 and $B(i,i';j,j')=(i,j)(i',j') - (i,j')(i',j) \in {\cal B}_0(S)$, then
 $\Vert Z \Vert_1$ can be reduced by ${\cal B}_0(S)$.
\end{lemma}

\begin{proof}
 Assume that the four blocks have the pattern of signs (i) and that 
 $$
 z_{i_b j_b} < 0, \quad 
 (i_b,j_b) \in {\cal I}_{k'l}, \quad 
 z_{i_c j_c} < 0, \quad 
 (i_c,j_c) \in {\cal I}_{kl'}. 
 $$
 We note that $(i_c,j_b) \in {\cal I}_{kl}$. 
 If $z_{i_c j_b} > 0$ or $z_{i_b j_c} > 0$, 
 $\Vert Z \Vert_1$ can be reduced by ${\cal B}_0(S)$.
 Suppose $z_{i_c j_b}= z_{i_b j_c} = 0$.
 Let $B=(i_b,j_b)(i_c,j_c) - (i_b,j_c)(i_c,j_b)$.
 Denote $Z' = \{z'_{ij}\}_{(i,j) \in {\cal I}} = Z+B$. 
 Then we have $z'_{i_c j_b} < 0$.
 Since there exists $(i_a, j_a) \in {\cal I}_{kl}$ such that 
 $z'_{i_a j_a} > 0$, 
 $Z'$ has both positive and negative cells in  ${\cal I}_{kl}$. 
 Hence $\Vert Z \Vert_1$ can be reduced by ${\cal B}_0(S)$ from Lemma
 \ref{lemma:reduction}.  Proofs for the other patterns are the same by symmetry.
\end{proof}

 \begin{figure}[htbp]
  \centering
  \includegraphics{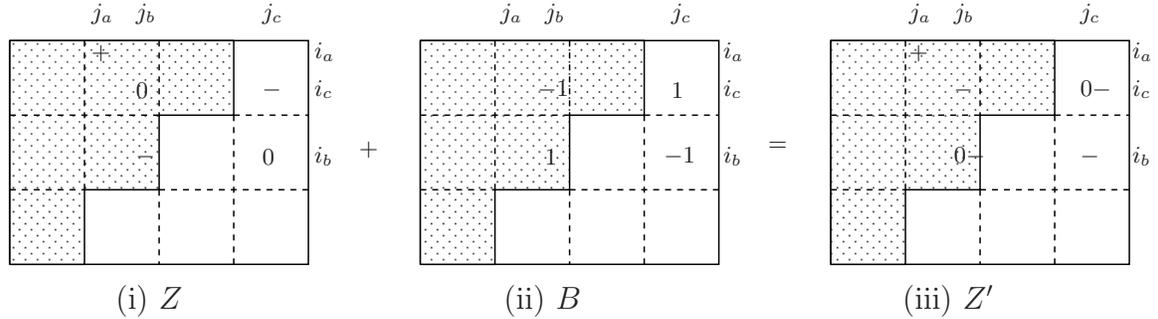}\\
  (i) $Z$ \hspace{4cm} (ii) $B$ \hspace{4cm} (iii) $Z'$
  \caption{$Z$ and $Z'$ when $z_{i_c j_b}= z_{i_b j_c} = 0$}
  \label{figure:lemma3}
 \end{figure}

\subsection{The sufficient condition for $2 \times 2$ block diagonal
  sets}\label{subsec3:4}
In this subsection we give a proof of the sufficient condition of
Theorem \ref{Th:main} when $S$ is equivalent to a $2 \times 2$ block
diagonal set.   

\begin{proposition}
 \label{prop:block}
 If $S$ is equivalent to a $2 \times 2$ block diagonal set, 
 ${\cal B}_0(S)$ is a Markov basis for $A_S$.
\end{proposition}

\begin{proof}
 Suppose that $Z \neq 0$. 
 If $Z$ contains a block ${\cal I}_{kl}$ which has both
 positive and negative cells,  $\Vert Z \Vert_1$ can be reduced by
 ${\cal B}_0(S)$ from Lemma \ref{lemma:reduction}.  

 Next we suppose that all four blocks are nonnegative or nonpositive. 
 Without loss of generality we can assume that ${\cal I}_{11} \ge 0$.
 Since all row sums and column sums of $Z$ are zeros, 
 we have ${\cal I}_{12} \le 0$, ${\cal I}_{21} \le 0$ and 
 ${\cal I}_{22} \ge 0$.
 On the other hand, since 
 $\sum_{(i,j) \in S} z_{ij} = 0$, 
 we have ${\cal I}_{22} \le 0$.
 However this implies $Z=0$ and  contradicts the assumption.
\end{proof}

\subsection{The sufficient condition for triangular sets}\label{subsec3:5} 
In this subsection we give a proof of the sufficient condition of
Theorem \ref{Th:main} when $S$ is equivalent to an
$n \times (n+1)$ triangular set in Figure \ref{figure:triangular}. 
We only need to consider this case if we allow transposition of the
tables and because of the fact that an $n \times n$ triangular set can
be considered 
as a special case of an  $n \times (n+1)$ triangular set as discussed
in Section \ref{sec:block-diagonal-trangular}.

In general, as mentioned, each block consists of more than one cell.
However for simplicity we first consider the case where every block consists of 
one cell. As seen at the end of this section,
actually it is easy to prove the sufficient condition of Theorem
\ref{Th:main} for general triangular set, 
once it is proved for the triangular sets with 
each block consisting of one cell.  Therefore the main result of this
section is the following proposition.

\begin{proposition}
 \label{prop:triangular-1}
 Suppose that $S$ is equivalent to an $n \times (n+1)$ triangular set
 in Figure \ref{figure:triangular} and every block consists of
 one cell.  
 Then ${\cal B}_0(S)$ is a Markov basis for $A_S$.
\end{proposition}

We prove this proposition based on a series of lemmas.  In all lemmas we
assume that $S$ is equivalent to an $n \times (n+1)$ triangular set.
If $n=1$, each fiber has only one element.  Hence we assume that $n
\ge 2$.

\begin{lemma}
 \label{lemma:3}
 If $Z$ contains a row $i_a$ such that the signs of $z_{i_a 1}$
 and $z_{i_a,n+1}$ are different,   then 
 $\Vert Z \Vert_1$ can be reduced by ${\cal B}_0(S)$.
\end{lemma}

\begin{proof}
 Without loss of generality we can assume that 
 $z_{i_a 1} > 0$ and $z_{i_a,n+1} < 0$. 
 Since $\sum_{i=1}^n z_{i,n+1}=0$, 
 there exists $i_b$ such that $z_{i_b,n+1} > 0$ as presented in 
 Figure \ref{figure:lemma4}.
 Hence if we set $B = (i_a,n+1)(i_b,1)-(i_a,1)(i_b,n+1)$,
 $B \in {\cal B}_0(S)$ and $\Vert Z + B \Vert_1 < \Vert B \Vert_1$. 
\end{proof}

\begin{figure}[htbp]
  \centering
  \includegraphics{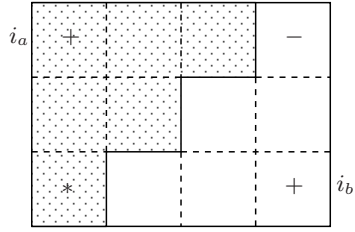}\\
  \caption{The case of $n=3$, $i_a=1$ and $i_b=3$}
  \label{figure:lemma4}
 \end{figure}

\begin{lemma}
 \label{lemma:4}
 Suppose that $Z$ has three rows $i_a < i_b < i_c$ satisfying  
 either of the following conditions, 
 \begin{enumerate}
  \item[{\rm (i)}]  $z_{i_a 1} > 0$, $z_{i_b 1} < 0$ and $z_{i_c 1} >
    0$;
  \item[{\rm (ii)}] $z_{i_a 1} < 0$, $z_{i_b 1} > 0$ and $z_{i_c 1} <
    0$;
 \end{enumerate}
 Then $\Vert Z \Vert_1$ can be reduced by ${\cal B}_0(S)$.
\end{lemma}

\begin{proof}
 It suffices to prove  the case of (i).
 Since $z_{i_b 1} < 0$, there exists $j$ such that 
 $2 \le j \le n+1$ and $z_{i_b j} >0 $.
 If $(i_b,j) \in S$ as presented in Figure \ref{figure:lemma5}-(i), 
 $B = (i_a,j)(i_b,1)-(i_a,1)(i_b,j) \in {\cal B}_0(S)$ 
 and 
 $\Vert Z + B \Vert_1 < \Vert Z \Vert_1$.
 If $(i_b,j) \notin S$ as presented in Figure \ref{figure:lemma5}-(ii), 
 $B' = (i_c,j)(i_b,1)-(i_c,1)(i_b,j) \in {\cal B}_0(S)$ 
 and 
 $\Vert Z + B' \Vert_1 < \Vert Z \Vert_1$.
\end{proof}

\begin{figure}[htbp]
  \centering
  \includegraphics{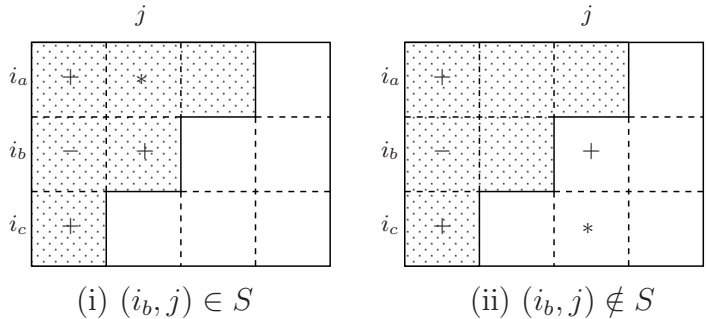}\\
  (i) $(i_b,j) \in S$ \hspace{2.6cm} (ii) $(i_b,j) \notin S$
  \caption{The case of $n=3$, $i_a=1$, $i_b=2$ and $i_c=3$}
  \label{figure:lemma5}
 \end{figure}

\begin{lemma}
 \label{lemma:5}
 Suppose that $Z$ contains four rows $i_a$, $i_b$, $i_c$ and $i_d$
 satisfying 
 $$
 z_{i_a 1} > 0,\quad z_{i_b,n+1} >0,\quad z_{i_c 1} < 0 \quad
 \mathrm{and}\quad
 z_{i_d, n+1} < 0
 $$
 and satisfying either of the following conditions, 
\[
{\rm (i)}\  i_a < i_c < i_b, \quad
{\rm (i')}\  i_a < i_d < i_b, \quad
{\rm (ii)}\  i_b < i_c < i_a, \quad
{\rm (ii')}\  i_b < i_d < i_a \ .
\]
Then $\Vert Z \Vert_1$ can be reduced by ${\cal B}_0(S)$.
\end{lemma}

\begin{proof}
 Suppose (i) $i_a < i_c < i_b$. 
 Since any row sum is zero, there exists $j$ such that 
 $z_{i_c j} > 0$ and $j \ge 2$.
 If $(i_c,j) \in S$ or $j = n+1$, 
 $B=(i_a, j)(i_c, 1) - (i_a, 1)(i_c, j) \in {\cal B}_0(S)$ and 
 $\Vert Z + B \Vert_1 < \Vert Z \Vert_1$ (Figure \ref{figure:lemma6-1}
 shows an example for this case). 
 Suppose that $(i_c,j) \in S^c$ and $j \neq n+1$.
 If $z_{i_c,n+1} < 0$, 
 $B =(i_b,j)(i_c,n+1)-(i_b,n+1)(i_c,j) \in {\cal B}_0(S)$ and 
 $\Vert Z + B \Vert_1 < \Vert Z \Vert_1$. 
 If $z_{i_c,n+1} = 0$, 
 $\Vert Z' \Vert_1 = \Vert Z + B \Vert_1 \le \Vert Z \Vert_1$
 and 
 $z'_{i_c,n+1} > 0$. 
 Since $z'_{i_a,1} > 0$ and $z'_{i_c,1} < 0$, 
 $\Vert Z_1 \Vert$ can be reduced by ${\cal B}_0(S)$.
 Hence $\Vert Z \Vert_1$ can be also reduced by ${\cal B}_0(S)$ (Figure
 \ref{figure:lemma6-2} shows an example for this case).

 \begin{figure}[htbp]
 \centering
 \includegraphics{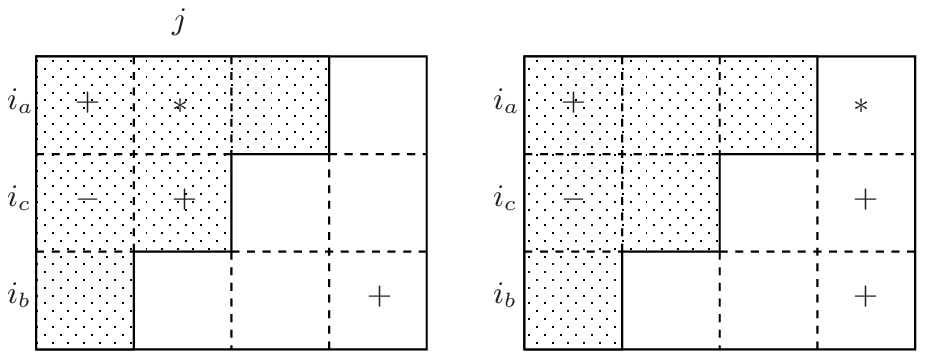}\\
 (i) $(i_c,j) \in S$ \hspace{2.6cm} (ii) $j = n+1$
 \caption{The case of $n=3$ and $i_a < i_c < i_b$}
 \label{figure:lemma6-1}
\end{figure}

 \begin{figure}[htbp]
 \centering
 \includegraphics{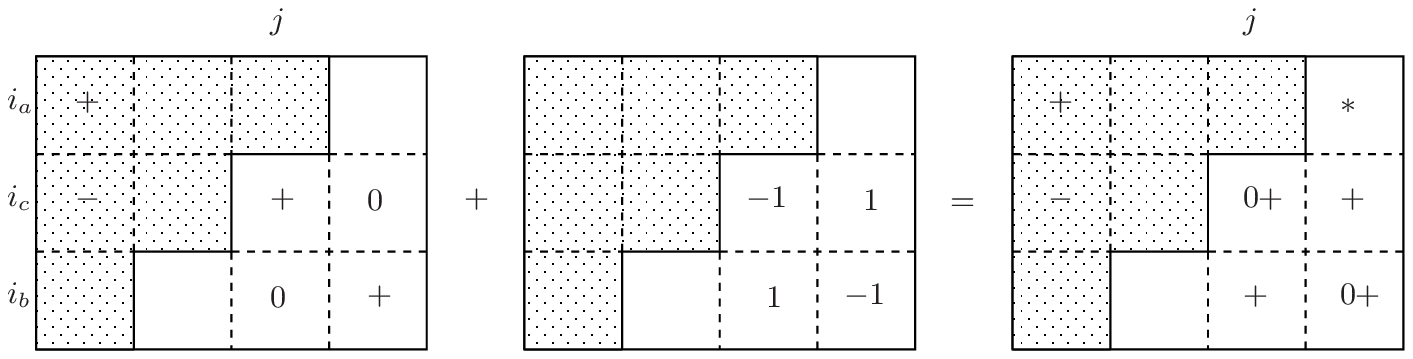}\\
  (i) $Z$ \hspace{3.6cm} (ii) $B$ \hspace{3.6cm} (iii) $Z'$
 \caption{The case of $n=3$ and $i_a < i_c < i_b$}
 \label{figure:lemma6-2}
\end{figure}

 \begin{figure}[htbp]
 \centering
 \includegraphics{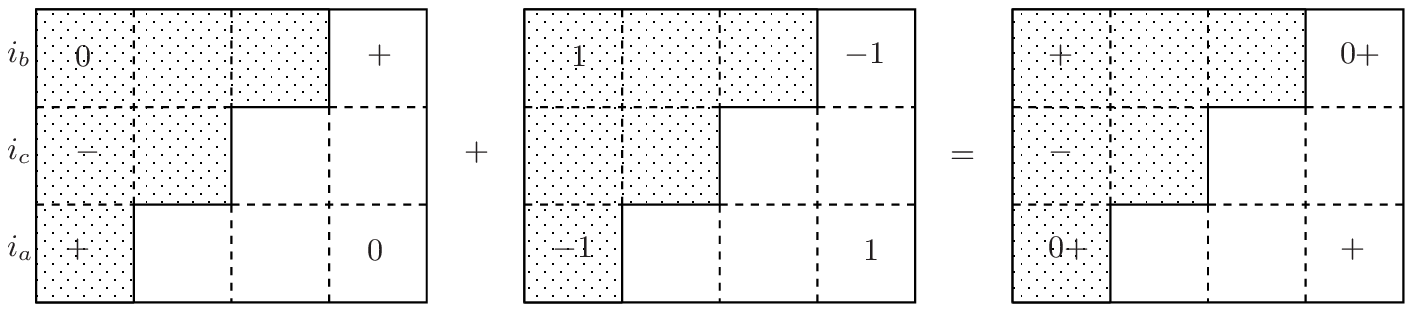}\\
  (i) $Z$ \hspace{3.6cm} (ii) $B$ \hspace{3.6cm} (iii) $Z'$
 \caption{The case of $n=3$ and $i_b < i_c < i_a$}
 \label{figure:lemma6-3}
\end{figure}

 In the case (i') $i_a < i_d < i_b$, we can prove the lemma
 in the same way by the symmetry of $n \times (n+1)$ 
 triangular pattern. 

 Suppose that (ii) $i_b < i_c < i_a$ or (ii') $i_b < i_d < i_a$.
 If $z_{i_a,n+1} < 0$ or $z_{i_b,1} < 0$, 
 the lemma holds from Lemma \ref{lemma:3}. 
 So we suppose that $z_{i_a,n+1} \ge 0$ and $z_{i_b,1} \ge 0$. 
 Let 
 $B = (i_a,n+1)(i_b,1) - (i_a, 1)(i_b, n+1)$.
 Then we have 
 $\Vert Z' \Vert = \Vert Z + B \Vert_1 \le \Vert Z \Vert_1$
 and 
 $z'_{i_a,n+1} > 0$, $z'_{i_b,1} > 0$. 
 Since $z'_{i_c,1}=z_{i_c,1} < 0$ and $z'_{i_d,n+1}=z_{i_d,n+1}<0$, 
 we can prove the lemma by applying the above argument (Figure
 \ref{figure:lemma6-3} shows an example for this case).
\end{proof}

From the definition of $L_1$-norm, $\Vert Z \Vert_1$ can be reduced
by ${\cal B}_0(S)$, if and only if $\Vert -Z \Vert_1$ can be reduced
by ${\cal B}_0(S)$. 
Thus without loss of generality we can assume that 
$z_{11} \ge 0$. 
{}From Lemma \ref{lemma:3}, \ref{lemma:4} and \ref{lemma:5}, 
it suffices to show that 
$\Vert Z \Vert_1$ can be reduced by ${\cal B}_0(S)$  
if  $Z$ satisfies
\begin{equation}
 \label{cond:i_0}
  \left\{
  \begin{array}{l}
   \exists i_0 \ge 1 \quad \mathrm{s.t.}\\
   z_{i1} \ge 0 \quad \mathrm{and} \quad z_{i,n+1} \ge 0 \quad
    \mathrm{for} \quad i \le i_0,\\
   z_{i1} \le 0 \quad \mathrm{and} \quad z_{i,n+1} \le 0 \quad
    \mathrm{for} \quad i > i_0, 
  \end{array}
  \right.
\end{equation}
as shown in Figure \ref{figure:exm-9}. 

 \begin{figure}[htbp]
 \centering
 \includegraphics{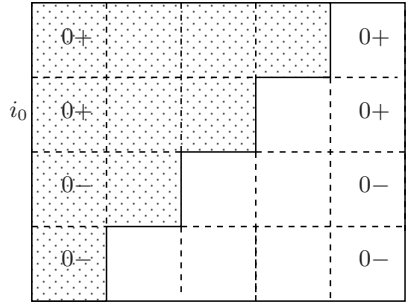}
 \caption{The case of $n=4$ and $i_0 = 2$}
 \label{figure:exm-9}
\end{figure}
We look at rows of  Figure  \ref{figure:exm-9} from the bottom and find the
last row $i_1$ such that at least one of 
$z_{i_1 1}$ or $z_{i_1, n+1}$ is negative, i.e., define $i_1$ by the
following conditions.
\begin{equation}
\label{eq:i1}
\begin{array}{ll}
  {\rm (i)} &  z_{i_1 1} \le 0  \text{ and } z_{i_1, n+1} \le 0;\\
  {\rm (ii)} &  z_{i_1 1} < 0 \text{ or } z_{i_1, n+1} < 0;\\
  {\rm (iii)} & z_{i1} = 0  \text{ and } z_{i,n+1}=0  \text{ for } i > i_1;
\end{array} 
\end{equation}
Note that if there exists no $i_1$ satisfying these conditions, then
the first column and  the last column of the table consists of only
zeros and we can use the induction on $n$.  Therefore 
for Lemmas  \ref{lemma:7}--\ref{lemma:8} below, 
we assume that there exists $i_1$ satisfying  (\ref{eq:i1}).
We also note that $i_1 > i_0$ when $i_1$ exists.

\begin{lemma}
\label{lemma:7}
Suppose $Z$ satisfies (\ref{cond:i_0}) and define $i_1$ 
by (\ref{eq:i1}) assuming that $i_1$ exists.
$\Vert Z \Vert_1$ can be reduced
if there exists 
$z_{ij} > 0$ for some $(i,j) \in S$ and $i \ge i_1$.
\end{lemma}

\begin{proof}
 Consider the case $z_{i_1 1}<0$.
 We note that there has to exist
 $i' < i_0$ such that $z_{i'1} > 0$
 from the condition (\ref{cond:i_0}). 
 Suppose that $i=i_1$. 
 Let $B$ be $B=(i_1, 1)(i', j) - (i_1, j)(i', 1)$.
 Then $\Vert Z + B \Vert_1 < \Vert Z \Vert_1$.
 Suppose that $i > i_1$ and $z_{i_1 j} \le 0$ for $(i_1,j) \in S$. 
 There has to exist $j'$ such that $z_{i_1 j'} > 0$.
 Let 
 $B=(i_1,j)(i,j')-(i_1,j')(i,j)$.
 If $z_{i_1 j} < 0$ or $z_{i j'} < 0$, 
 $\Vert Z' \Vert_1 = \Vert Z + B \Vert_1 
 < \Vert Z \Vert_1$.
 If $z_{i_1 j} = 0$ or $z_{i j'} = 0$, 
 $\Vert Z' \Vert_1 = \Vert Z + B \Vert_1 
 \le \Vert Z \Vert_1$.
 As shown in Figure \ref{figure:exm-10}, 
 since 
 $z'_{i'1} > 0$, $z'_{i_0 1} < 0$ and $z'_{i_0 j} >0$, 
 $\Vert Z' \Vert$ can be reduced by ${\cal B}_0(S)$. 
 Hence $\Vert Z \Vert_1$ can also be reduced by ${\cal B}_0(S)$.

 Next we consider the case $z_{i_1,n+1}<0$.
 Then there has to exist  $i' < i_0$ such that $z_{i',n+1} > 0$
 from the condition (\ref{cond:i_0}). 
 When $i=i_1$, let 
 $B=(i_1,n+1)(i',j)-(i_1,j)(i',n+1)$.
 Then $\Vert Z + B \Vert_1 < \Vert Z \Vert_1$.
 When $i > i_1$ and $z_{i_1 j} \le 0$ for $(i_1,j) \in S$, 
 a similar proof to the case $z_{i_1 1}<0$ 
 can be given as shown in Figure \ref{figure:exm-10-2}

\end{proof}
 \begin{figure}[htbp]
 \centering
  \includegraphics[scale=0.6]{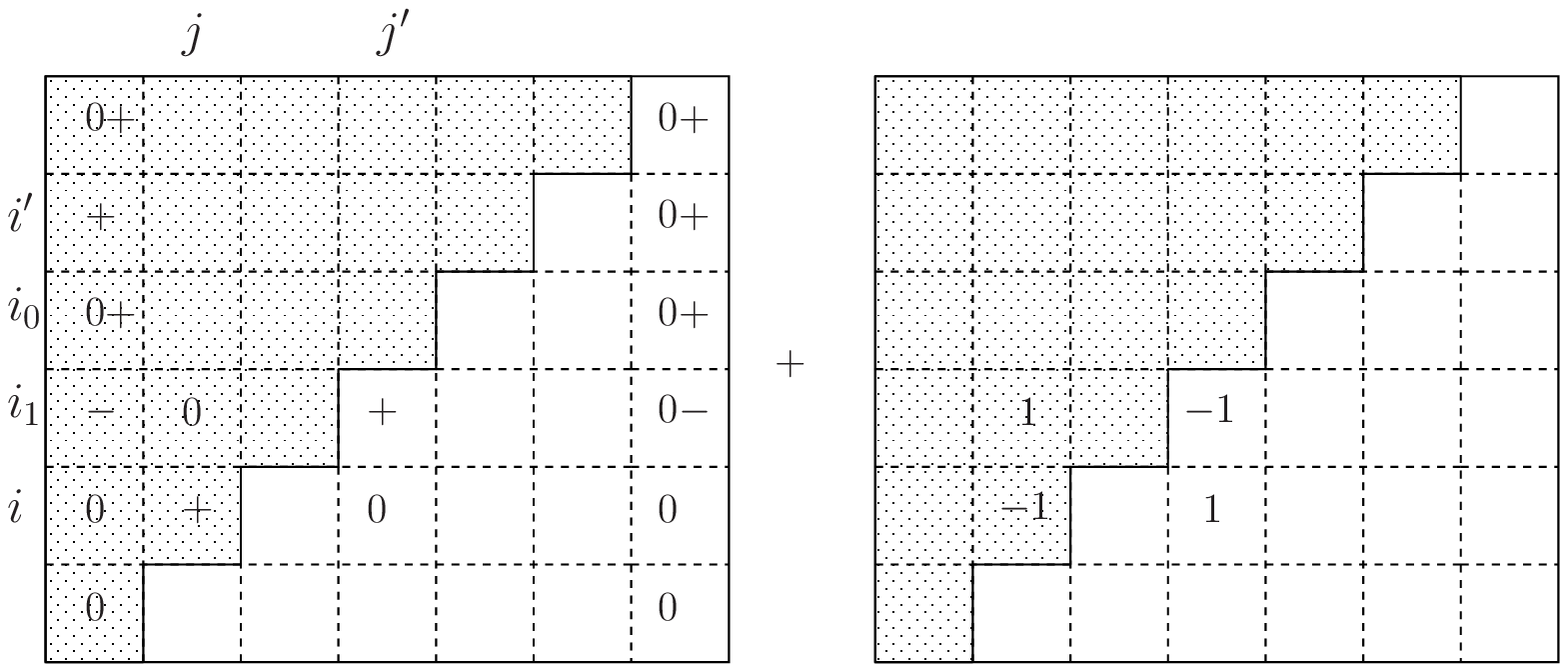}
  \includegraphics[scale=0.6]{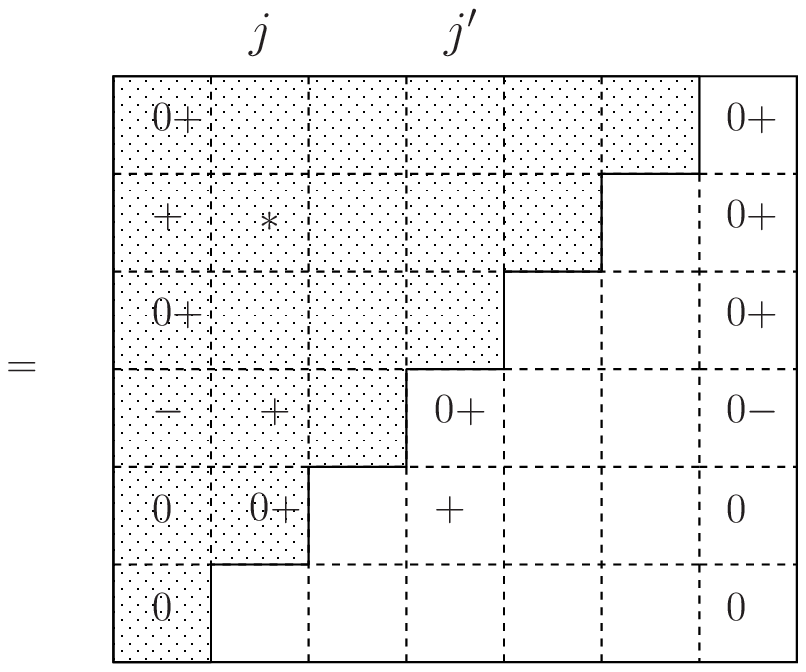}\\
  (i) $Z$ \hspace{3.6cm} (ii) $B$ \hspace{3.6cm} (iii) $Z'$
 \caption{The case of $n=6$, $(i_0,i_1,i,i',j,j') = (3,4,5,2,2,4)$ and
  $z_{i_1 1} < 0$}
 \label{figure:exm-10}
\end{figure}

 \begin{figure}[htbp]
 \centering
  \includegraphics[scale=0.6]{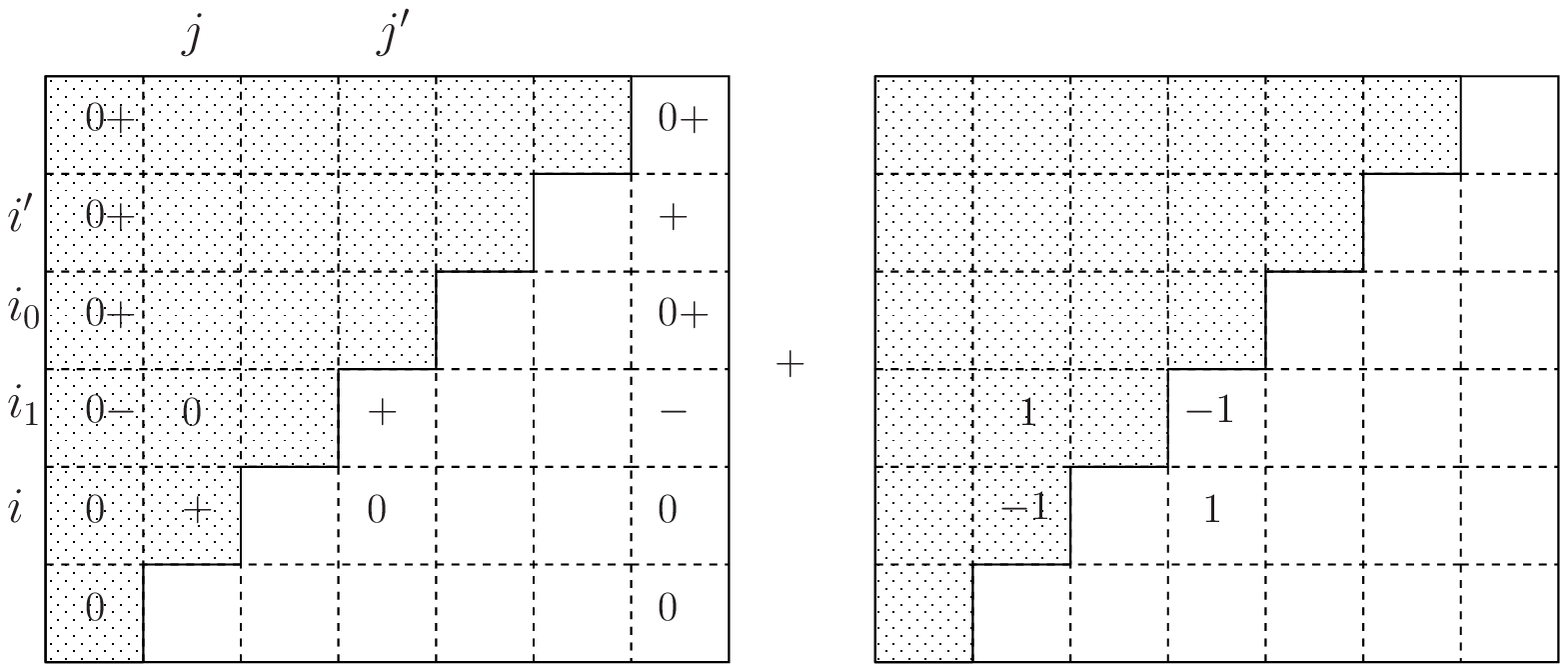}
  \includegraphics[scale=0.6]{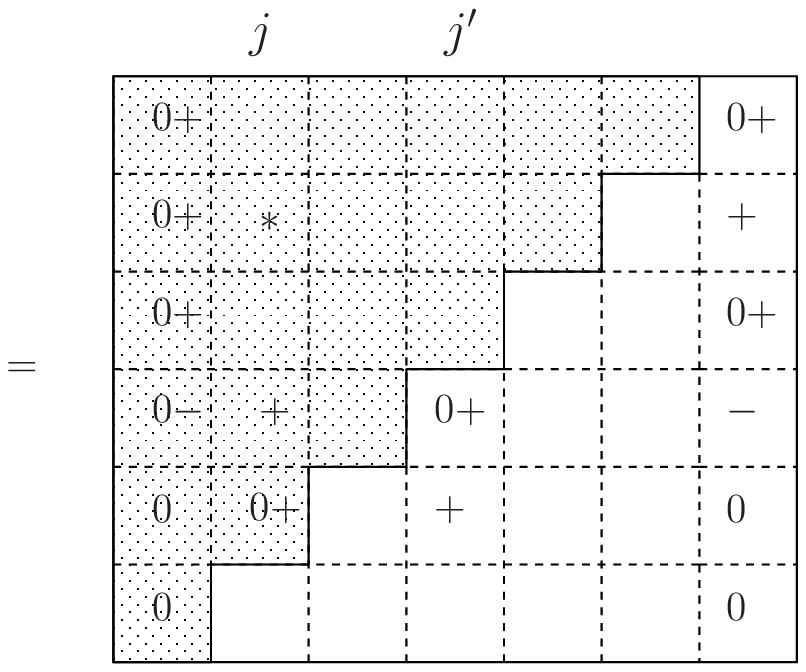}\\
  (i) $Z$ \hspace{3.6cm} (ii) $B$ \hspace{3.6cm} (iii) $Z'$
 \caption{The case of $n=6$, $(i_0,i_1,i,i',j,j') = (3,4,5,2,2,4)$ and
  $z_{i_1,n+1} < 0$}
 \label{figure:exm-10-2}
\end{figure}

We define some more sets.
Let $\bar S^c$ and $\bar S^c_i$, $i=2,\ldots,n$, 
be the sub-triangular set of $S^c$
defined as
$$
\bar S^c = \{(i',j') \in S^c \mid j' \neq n+1\}, \quad
\bar S^c_i = \{(i',j') \in S^c \mid i' < i, \; j' \neq n+1\}, 
$$
respectively. 
Figure \ref{figure:barS} shows $\bar S^c$ and $\bar S^c_i$ for 
$n=4$, $i=4$. 
We note that 
\begin{equation}
 \label{eq:sum_Sbar}
  \sum_{(i,j) \in \bar S^c} z_{ij} = 0
\end{equation}
for all $Z$, because the last column sum is zero and 
$\sum_{(i,j) \in S^c} z_{ij} = 0$.
We also define $\bar S^-_i$, $i=2,\ldots,n$, by 
$$
\bar S^-_i = \{(i',j') \in \bar S^c_i \mid z_{i'j'} < 0\} .
$$

 \begin{figure}[htbp]
 \centering
 \includegraphics{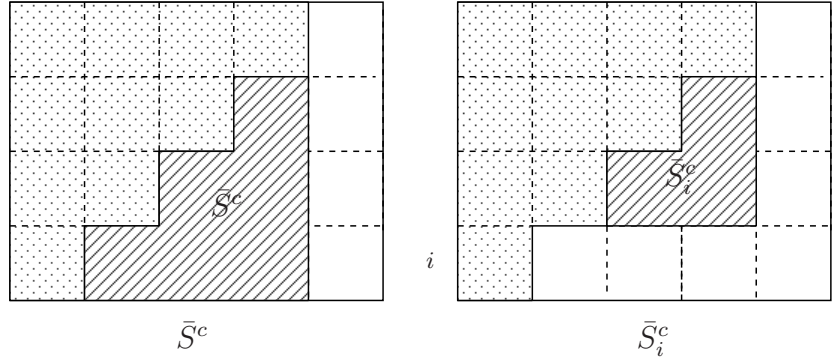}\\
  \hspace*{0.5cm}$\bar S^c$\hspace{5.7cm}$\bar S_i^c$
 \caption{$\bar S^c$ and $\bar S^c_i$ for $n=4$ and $i = 4$}
 \label{figure:barS}
\end{figure}

{}From Lemma \ref{lemma:7} it suffices to consider $Z$ such that
$z_{ij} \le  0$ for all $(i,j) \in S$ and $i \ge i_1$.  The following
lemma states a property of such a $Z$.

\begin{lemma}
 \label{lemma:6}
 Suppose that $Z$ satisfies (\ref{cond:i_0}) and 
define $i_1$ by (\ref{eq:i1}) assuming that $i_1$ exists.
Furthermore assume that 
$z_{ij} \le 0$ for all $(i,j) \in S$, $i \ge i_1$.
Then 
 \begin{equation}
  \label{ineq:2}
  \vert z_{i_1 1} + z_{i_1, n+1} \vert \le 
  \sum_{(i,j) \in \bar S^-_{i_1}} \vert z_{ij} \vert. 
 \end{equation}
\end{lemma}

\begin{proof}
 Assume that 
 $$
 \vert z_{i_1 1} + z_{i_1, n+1} \vert  > 
 \sum_{(i,j) \in \bar S^-_{i_1}} \vert z_{ij} \vert. 
 $$  
 Since the roles of $z_{i_1 1}$ and $z_{i_1, n+1}$ are interchangeable,
 we assume that $|z_{i_1 1}| > 0$.
 Then 
 there exist nonnegative integers $w^1_{ij}$, $w^{n+1}_{ij}$ and 
 the set of cells $S' \subseteq \bar S^-_{i_1}$ and 
 $S'' \subseteq \bar S^-_{i_1}$ satisfying
 $$
 w^1_{ij} + w^{n+1}_{ij} \le \vert z_{ij} \vert, 
 \quad 
 \sum_{(i,j) \in S'} w^1_{ij} +
 \sum_{(i,j) \in S''} w^{n+1}_{ij} 
 =  \sum_{(i,j) \in \bar S^-_{i_1}} \vert z_{ij} \vert. 
 $$
 $$
 \sum_{(i,j) \in S'} w^1_{ij} < \vert z_{i_1 1} \vert,\quad
 \sum_{(i,j) \in S''} w^{n+1}_{ij} \le \vert z_{i_1,n+1} \vert.
 $$
 $S'$ and $S'{}'$ may have overlap if  $|z_{ij}|\ge 2$ for some
 cell $(i,j) \in \bar S^-_{i_1}$.
 For $(i,j) \in \bar S^-_{i_1}$, 
 let $B^1(i,j)$ and $B^{n+1}(i,j)$ be defined by 
 $$
   B^1(i,j) = (i,j)(i_1,1)-(i,1)(i_1,j), \quad 
   B^{n+1}(i,j)=(i,j)(i_1,n+1)-(i,n+1)(i_1,j).
 $$
 We note that 
 $B^1(i,j) \in {\cal B}_0(S)$ and $B^{n+1}(i,j) \in {\cal B}_0(S)$ for 
 any $(i,j) \in \bar S^-_{i_1}$. 
 Denote 
 \begin{equation}
  \label{def:z'}
   Z^{\prime} = \{z^{\prime}_{ij}\}_{(i,j) \in {\cal I}} 
   =Z + \sum_{(i,j) \in S'} w^1_{ij} B^1(i,j)
   +
   \sum_{(i,j) \in S''} w^{n+1}_{ij} B^{n+1}(i,j).
 \end{equation}
 Then we have 
 $z^{\prime}_{i_1 1} < 0$, $z^{\prime}_{i_1,n+1} \le 0$, 
 and  $z^{\prime}_{ij} \ge 0$ for all
 $(i,j) \in \bar S^c_{i_1}$. 
 This implies that 
 \begin{equation}
  \label{leq:i_1}
   \sum_{(i,j) \in \bar S^c \setminus \bar S^c_{i_1}} z^{\prime}_{ij}
   \le 0.
 \end{equation}
 On the other hand from the condition of Lemma \ref{lemma:6}
\[
\sum_{(i,j) \in \bar S^c \setminus \bar S^c_{i_1}} z^{\prime}_{ij}
= \sum_{i = i_1}^n \sum_{j=1}^{n+1} 
z^{\prime}_{ij} - \big( \sum_{i\ge i_1,\; (i,j)\in S} z^{\prime}_{ij}
 + \sum_{i=i_1}^n z^{\prime}_{i,n+1}\big) > 0,
\]
 which contradicts (\ref{leq:i_1}) (See Figure \ref{figure:exm-11}).
\end{proof}
\begin{figure}[htbp]
 \centering
  \includegraphics[scale=0.7]{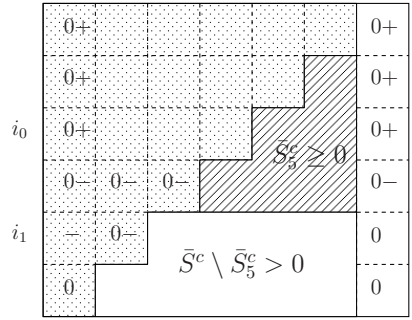}
 \caption{The case of $n=6$ and $(i_0,i_1) = (3,5)$}
 \label{figure:exm-11}
\end{figure}

\begin{lemma}
 \label{lemma:8}
 Suppose that $Z$ satisfies (\ref{cond:i_0}) and the conditions of
 Lemma \ref{lemma:6}. Then 
 \begin{enumerate}
  \item[{\rm (i)}] $i_1 \ge 3$;
  \item[{\rm (ii)}] If $i_1 = 3$, $\Vert Z \Vert_1$ can be reduced.
 \end{enumerate}
\end{lemma}

\begin{proof}
 (i)\; It is obvious that $i_1 \ge 2$. 
 Suppose $i_1 =2$. 
 Since any row sum of $Z$ is zero, 
 we have 
 $$
 \sum_{(i,j) \in \bar S^c} z_{ij} > 0, 
 $$
 from (ii) and (iii) of (\ref{eq:i1}).
 However this contradicts (\ref{eq:sum_Sbar}).\\
 (ii)\;  When $i_1=3$, $\bar S_3 = \{(2,n)\}$ and $z_{2n} < 0$ 
 from Lemma \ref{lemma:6}. 
 If $z_{21} > 0$, we have $z_{31} < 0$ from 
 (iii) of (\ref{eq:i1}).
 Therefore $B=(3,1)(2,n)-(2,1)(3,n)$ satisfies
 $\Vert Z + B\Vert_1 < \Vert Z \Vert_1$ (Figure \ref{figure:exm-12}-(i)).
 If $z_{2,n+1} > 0$, $z_{3,n+1} < 0$ 
 from (iii) of (\ref{eq:i1}).
 Hence $B=(2,n)(3,n+1)-(2,n+1)(3,n)$
 satisfies
 $\Vert Z + B\Vert_1 < \Vert Z \Vert_1$(Figure \ref{figure:exm-12}-(ii)).
 Next we consider the case of $z_{21}=z_{2,n+1} = 0$. 
 Then $z_{11} > 0$ or $z_{1,n+1} > 0$.
 Suppose that $z_{11} > 0$.
 This implies $z_{31} < 0$.
 Since $z_{2n} < 0$, 
 there exists $2 \le j \le n-1$ such that $z_{2j} > 0$.
 Then if we set 
 $B=(1,j)(2,1)-(1,1)(2,j)$, 
 we have 
 $\Vert Z' \Vert_1 = \Vert Z + B \Vert_1 \le \Vert Z \Vert_1$
 and 
 $z'_{21} > 0$, $z'_{31} < 0$ and $z'_{2n} < 0$ (Figure \ref{figure:exm-13}).
 Then $\Vert Z' \Vert_1$ can be reduced by 
 ${\cal B}_0(S)$. 
 Therefore $\Vert Z \Vert_1$ can be also reduced by 
 ${\cal B}_0(S)$. 
When $z_{1,n+1} > 0$, a proof is similar.
\end{proof}

\begin{figure}[htbp]
\centering
 \includegraphics{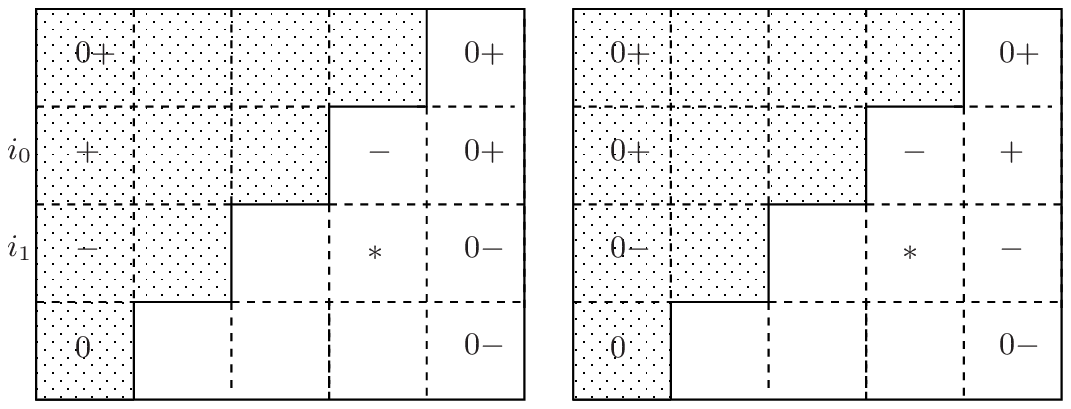}\\
  \hspace*{0.7cm}(i) $z_{21} > 0$ \hspace{3.3cm}(ii) $z_{25} > 0$
 \caption{The case of $n=4$ and $z_{21}>0$ 
 or $z_{25}>0$}
 \label{figure:exm-12}
\end{figure}

 \begin{figure}[htbp]
 \centering
  \includegraphics[scale=0.9]{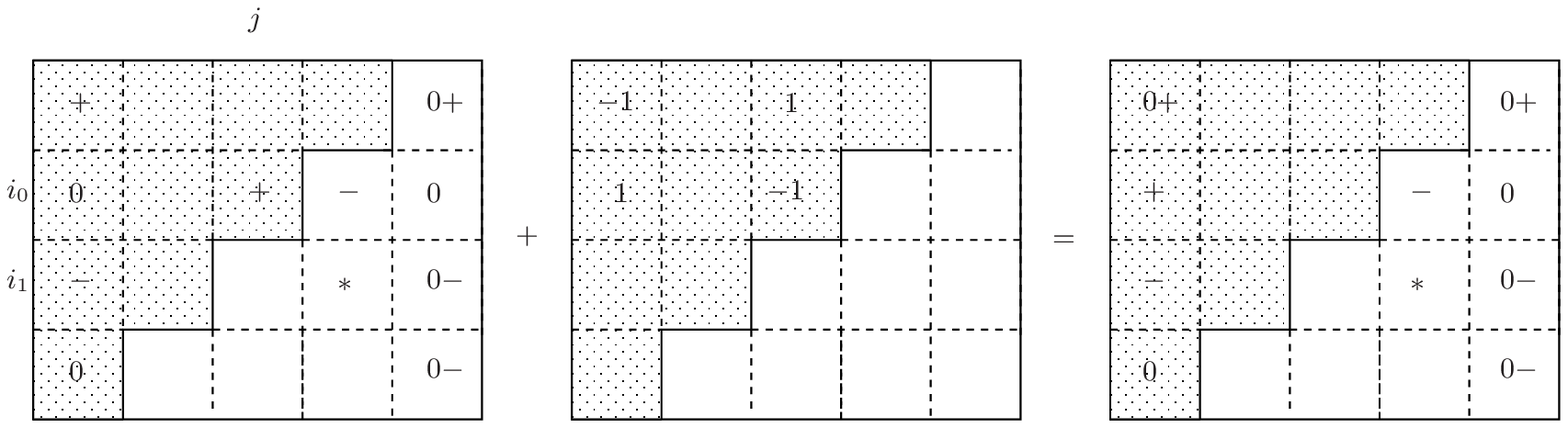}\\  
  (i) $Z$ \hspace{4.2cm} (ii) $B$ \hspace{4cm} (iii) $Z'$
 \caption{The case of $n=4$, $z_{21}=0$ and $z_{25}=0$}
 \label{figure:exm-13}
\end{figure}

By using Lemmas \ref{lemma:3}--\ref{lemma:8}, 
we give a proof of Proposition \ref{prop:triangular-1}.

\begin{proof}[Proof of Proposition \ref{prop:triangular-1}]
We prove this proposition by the induction on the number of rows $n$.
Suppose that $n=2$. 
Then 
$$
z_{1j} + z_{2j} = 0 \quad \mbox{for} \quad j=1,2,3.
$$
$$
z_{11} + z_{12} + z_{21} = 0, \quad 
z_{22} + z_{13} + z_{23} = 0.
$$
Hence $z_{12} = z_{22} =0$. 
Therefore $Z$ is equivalent to a move in 
the $2 \times 2$ pattern as in Figure \ref{figure:2x2} with 
fixed row sums and column sums.
It is easy to see that that proposition holds 
for this pattern.

\begin{figure}[htbp]
 \centering
 \includegraphics[scale=0.7]{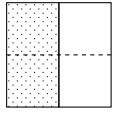}
 \caption{A $2 \times 2$ pattern}
\label{figure:2x2}
\end{figure}

Suppose $n > 2$ and assume that this proposition holds 
for triangular sets smaller than $n \times (n+1)$.
From the results of Lemmas \ref{lemma:3}--\ref{lemma:8}, 
it suffices to show that if $Z$ satisfies (\ref{cond:i_0}) and 
the conditions of Lemma \ref{lemma:6}, 
$\Vert Z \Vert_1$ can be reduced by ${\cal B}_0(S)$. 
We prove this by the induction on $i_1$. 

Suppose that $i^*_1 > 3$ and assume that 
$Z$ with $i_1 < i^*_1$ 
can be reduced by ${\cal B}_0(S)$. 
From Lemma \ref{lemma:6}, (\ref{ineq:2}) holds.
Thus 
there exist nonnegative integers $w^1_{ij}$, $w^{n+1}_{ij}$ and 
the set of cells $S' \subseteq \bar S^-_{i_1}$ and 
$S'' \subseteq \bar S^-_{i_1}$ satisfying
$$
w^1_{ij} + w^{n+1}_{ij} \le \vert z_{ij} \vert, 
\quad 
\sum_{(i,j) \in S'} w^1_{ij} = \vert z_{i_1 1} \vert,\quad
\sum_{(i,j) \in S''} w^{n+1}_{ij} = \vert z_{i_1,n+1} \vert.
$$

Let $Z'$ be defined as in (\ref{def:z'}).
Then we have $\Vert Z' \Vert_1 \le \Vert Z \Vert_1$.
If $\Vert Z' \Vert_1 < \Vert Z \Vert_1$, 
this proposition holds.
Suppose that $\Vert Z' \Vert_1 = \Vert Z \Vert_1$.
Then $Z'$ satisfies either of the following three conditions, 
\begin{enumerate}
 \item[(i)] $z^{\prime}_{i1}=0$ and $z^{\prime}_{i,n+1}=0$ for
     $i=1,\ldots,n$; 
 \item[(ii)] there exists $i$ such that 
     $z^{\prime}_{i1} \neq 0$ or $z^{\prime}_{i,n+1} \neq 0$ 
     and $Z'$ does not satisfy (\ref{cond:i_0}).
 \item[(iii)] there exists $i$ such that 
     $z^{\prime}_{i1} \neq 0$ or $z^{\prime}_{i,n+1} \neq 0$ 
     and $Z'$ satisfies (\ref{cond:i_0}).
\end{enumerate}
In the case of (i), $\Vert Z' \Vert_1$ can be reduced by 
${\cal B}_0(S)$ from the inductive assumption on $n$.
In the case of (ii), $\Vert Z' \Vert_1$ can be reduced by
by Lemma \ref{lemma:5}. 
In the case of (iii), 
noting that $z^{\prime}_{i1}=0$ and $z^{\prime}_{i,n+1}=0$ for 
$i \ge i_1$, 
$\Vert Z' \Vert_1$ can be reduced from the
inductive assumption on $i_1$.
\end{proof}

So far we have given a proof when every block has only 
one cell.  It remains to extend Proposition \ref{prop:triangular-1}
to general triangular sets.  Based on the results of Lemma
\ref{lemma:reduction} 
and \ref{lemma:2}, we see that Proposition \ref{prop:triangular-1}
can be extended to the case where $n \ge 2$.
Then it suffices to consider
the case of $n=1$ as in Figure \ref{figure:1x2}.

\begin{figure}[htbp]
 \centering
 \includegraphics{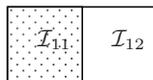}
 \caption{$1 \times 2$ triangular pattern}
\label{figure:1x2}
\end{figure}

\begin{lemma}
Suppose that $S$ is equivalent to a $1 \times 2$ triangular set. 
 Then ${\cal B}_0(S)$ is a Markov basis for $A_S$.
\end{lemma}

\begin{proof}
 Since $Z$ satisfies $\sum_{(i,j) \in S} z_{ij}=0$ and 
 $\sum_{(i,j) \in S^c} z_{ij}=0$, 
 $Z\neq 0$ has to contain both positive and negative cells in  
 at least one of ${\cal I}_{11}$ and ${\cal I}_{12}$.
 Hence $\Vert Z \Vert_1$ can be reduced from Lemma \ref{lemma:reduction}.
\end{proof}

Now we have completed a proof of the sufficient condition of Theorem
\ref{Th:main} for general triangular set $S$.

\section{Concluding remarks}
\label{sec:discussion}

In this paper we consider Markov bases consisting of square-free
moves of degree two for two-way subtable sum problems.  
We gave a necessary and sufficient condition for the existence
of a Markov basis consisting of square-free moves of degree two.

From our results, if $S$ contains a pattern ${\cal P}$ or ${\cal P}^t$,   
a Markov basis has to include a move with degree higher than or equal
to four. From theoretical viewpoint, it is interesting to study the structure
of Markov bases for such cases. 
Our results may give insights into the problem. 
However it seems difficult at this point and left to our future
research. 

Consider a particular fiber with $x(S)=0$ in the subtable sum
problems.  Then $x_{ij}=0$ for all $(i,j)\in S$.  This implies that
this fiber is also a fiber for a problem where all cells of $S$ are
structural zeros.  Therefore Markov bases for the subtable sum
problems for $S$ are also Markov bases for a problem where all cells
of $S$ are structural zeros.  Various properties of Markov bases are
known for structural zero problems.  It is of interest to investigate
which properties of Markov bases for structural zero problem for $S$
can be generalized to subtable sum problem for $S$.

Ohsugi and Hibi have been investigating properties of Gr\"obner bases
arising from finite graphs (\cite{ohsugi-hibi-1999aam},
\cite{ohsugi-hibi-1999ja}, \cite{ohsugi-hibi-2005jaa}). 
With bipartite graphs, their
problem is equivalent to two-way contingency tables with structural
zeros.  
From the viewpoint of graphs of Ohsugi and Hibi, the subtable sum
problem corresponds to a complete bipartite graph with two kinks of
edges. 
It would be also very interesting to investigate  subtable sum problem
from the viewpoint of Gr\"obner bases.

We used the norm reduction argument to prove that ${\cal B}_0(S)$ is a
Markov basis.  It should be noted that ${\cal B}_0(S)$ for the
subtable sum problem is not necessarily 1-norm reducing {\em in one
  step}, even when ${\cal B}_0(S)$ is the unique minimal Markov basis.
Therefore the subtable sum problem is worth to be considered from 
the viewpoint of norm reduction by a Markov basis.

\bibliographystyle{plainnat}
\bibliography{Hara-Takemura-Yoshida}
\end{document}